\documentclass{rmmcart}
\usepackage{amsmath, amssymb, amsthm}
\usepackage{latexsym}
\usepackage{mathtools}
\usepackage{ulem}
\usepackage{graphicx}
\usepackage{color}
\usepackage{url}
\usepackage[mathscr]{eucal}
\numberwithin{equation}{section}
\date{}
\DeclareMathAlphabet{\itbf}{OML}{cmm}{b}{it}

\newcommand{\RR}{\mathbb{R}}
\newcommand{\CC}{\mathbb{C}}
\newcommand{\NN}{\mathbb{N}}

\newcommand{\ds}{\displaystyle}
\newcommand{\no}{\nonumber}
\newcommand{\ri}{\rightarrow}

\newcommand{\bS}{\textsl{\textbf{S}}}
\newcommand{\bD}{\textsl{\textbf{D}}}

\newcommand{\q}{\quad}

\newcommand{\SK}{{{\mathscr{K}}}}

\newcommand{\bx}{{\itbf x}}
\newcommand{\x}{{\itbf x}}
\newcommand{\by}{{\itbf y}}
\newcommand{\y}{{\itbf y}}

\newcommand{\e}{\boldsymbol{e}}
\newcommand{\bw}{{\itbf w}}
\newcommand{\w}{{\itbf w}}
\renewcommand{\i}{\mathrm{i}}

\newcommand{\bn}{{\itbf n}}
\newcommand{\n}{{\itbf n}}

\newcommand{\bH}{{\itbf H}}

\newcommand{\bi}{\begin{itemize}}
\newcommand{\ei}{\end{itemize}}

\renewcommand{\i}{\mathrm{i}}

\newcommand{\bE}{{\itbf E}}

\newcommand{\bK}{\itbf K}

\newcommand{\be}{\begin{eqnarray}}
\newcommand{\ee}{\end{eqnarray}}

\newcommand{\ben}{\begin{eqnarray*}}
\newcommand{\een}{\end{eqnarray*}}
\def\ds{\displaystyle}

\newtheorem{thm}{Theorem}[section]
\newtheorem{lem}[thm]{Lemma}
\newtheorem{prop}[thm]{Proposition}

\newcommand{\bea}{\begin{eqnarray*}}
\newcommand{\eea}{\end{eqnarray*}}
\newcommand{\bean}{\begin{eqnarray}}
\newcommand{\eean}{\end{eqnarray}}
\newcommand{\p}{\partial}
\newcommand{\f}{\frac}

\newcommand{\surface}{\p D}

\newcommand{\divs}{\mathop{\mathrm{div}_{\surface}}}

\newcommand{\curl}{\boldsymbol{\mathrm{curl}} \; }

\newcommand{\gradx}{\boldsymbol{\mathrm{grad}}_{\x}  }

\newcommand{\SSS}{{{\mathscr{S}}}}
\newcommand{\tot}{\mathrm{tot}}
\newcommand{\inc}{\mathrm{inc}}

\newcommand{\Mr}{\mathbb{M}}

\newcommand{\Mm}{\boldsymbol{\Mr}}

\newcommand{\Ir}{\mathbb{I}}
\newcommand{\Ima}{\boldsymbol{\Ir}}
\newcommand{\Jr}{\mathbb{J}}
\newcommand{\Jm}{\boldsymbol{\Jr}}

\newcommand{\E}{\boldsymbol{E}}

\renewcommand{\H}{\boldsymbol{H}}

\newcommand{\h}{\boldsymbol{h}}

\newcommand{\zero}{\boldsymbol{0}}

\title{A well-posed surface currents and charges
  system for electromagnetism in dielectric media}

     \author{M. Ganesh}
     \address{Applied Mathematics and Statistics Department, Colorado School of Mines, 1500 Illinois St.,  Golden, CO 80401.}
     \email{mganesh@mines.edu}
   
	\author{S. C. Hawkins}
	\address{Department of Mathematics and Statistics, Macquarie University, Sydney, NSW 2109, Australia.}
\email{stuart.hawkins@mq.edu.au}

     \author{C. Jeznach}
     \address{Department of Mathematical Sciences, Worcester Polytechnic Institute, Worcester, MA 01609.}
     \email{cjeznach@wpi.edu}
    
		\author{D. Volkov}
     \address{Department of Mathematical Sciences, Worcester Polytechnic Institute, Worcester, MA 01609.}
     \email{darko@wpi.edu}
		\thanks{D. Volkov is supported by a Simons Collaboration Grant for Mathematicians.}
		
     \date{October, 2, 2018}

\begin{document}

\begin{abstract}
The  free space Maxwell dielectric problem 
can be reduced to a  system of surface integral equations (SIE). 
A numerical formulation for  the  Maxwell dielectric problem using an 
SIE system presents two key advantages: first,
 the  radiation condition at infinity is exactly satisfied, and second, 
there is no need to artificially define a  truncated  domain.
Consequently,  these SIE systems have generated much interest in physics, electrical engineering,
and mathematics, and many SIE  formulations have been proposed over time.
  In this article we introduce a new SIE 
formulation which is in the desirable operator form identity
plus compact, is well-posed, and remains well-conditioned  as the frequency tends to zero.
The unknowns in the formulation are three dimensional vector fields on the boundary of the dielectric body.
The SIE  studied in this paper is derived from a formulation
developed in earlier work by  some of the authors~\cite{ganesh2014all}. Our initial formulation
utilized linear constraints to obtain a uniquely solvable system for all frequencies. 
The new SIE introduced and analyzed in this article  combines the integral equations from \cite{ganesh2014all} with
new constraints. 
We show that the new system is in the operator form identity plus compact  in a
particular functional space, and we prove   well-posedness  at all frequencies and low-frequency stability of the new SIE.

\end{abstract}
  \maketitle
% \newpage
\section{Introduction}
Understanding the propagation of electromagnetic waves in three dimensional dielectric media is a fundamental 
problem which is central to a plethora of 
 applications~\cite{ammarikangbook, bsu:electromagnetic, chew:book,  colton:integral, coltoninverse, kerker:light, 
mish2006:book, mishchenko:book,  mishchenko:review,  muller:book, nedlec:book, stratton:book}. For certain
simple shaped dielectric bodies, the propagation process can be modeled by analytical 
techniques~\cite{bsu:electromagnetic}. However, in general, 
numerical modeling~\cite{chew:book, mish2006:book} of  the  
Maxwell system~\cite{muller:book, nedlec:book, stratton:book} 
in unbounded  three dimensional dielectric media is required. Direct simulation of the time-harmonic Maxwell system 
(using finite-difference/elements) requires truncation of the unbounded 
medium to a bounded region, and  approximation of the  associated Silver-M\"uller radiation condition.
Fortunately, these two key physical aspects of the Maxwell system can be preserved in numerical
modeling  by first reformulating the model as an equivalent system of surface integral equations (SIE) defined on the boundary of the dielectric body and approximating the SIE system in a finite dimensional space. 

There is a large literature spanning over a century on developing SIE reformulations of 
Maxwell equations~\cite{chew:book,  colton:integral, coltoninverse,  muller:book, nedlec:book, stratton:book}. 
The extensive survey in~\cite{epsgre2009pap} reviews a century  of research on deriving 
and using such SIEs. The key emphasis in the survey in~\cite{epsgre2009pap} is to highlight the lack of 
low-frequency stable SIE reformulation of the Maxwell equations that is also in the desirable second-kind (identity
plus compact operator) form. The SIE formulation in~\cite{epsgre2009pap} requires surface differential equations constraints.
The main focus of our work is to develop a low-frequency stable and resonance-free well-posed system for all frequencies
that involves only surface integral operators, and mathematically establish the robustness of the system with proofs in appropriate
function spaces. For electromagnetism, appropriate function spaces play a crucial role~\cite{costabel2011lect} for simulating  physically correct solutions.

In previous work~\cite{ganesh2014all}
we introduced an original SIE reformulation of the Maxwell dielectric system,
which satisfies  
a few remarkable properties that, when taken together,
 are unique: first, this system is in the classical form identity plus compact. 
Second, it does not suffer from spurious eigenfrequencies. Third, 
it does not suffer from the low-frequency breakdown phenomenon: as the frequency tends to zero, this system remains uniformly well posed (see \cite{YlaTas:dilow} for a simple account 
of the low-frequency breakdown). 
Fourth, the solution to this system is equal to the trace of the electric and magnetic fields
on the dielectric. Fifth, applying numerical methods developed
in~\cite{ghfirst}, we achieved spectral numerical convergence of discrete solutions.
Points one to four were proved in~\cite{ganesh2014all}.
We explained in~\cite{ganesh2014all} that well-posedness
at all frequencies is achieved when our ``square system'' is augmented by a linear surface integral based constraint.
In this article we propose a new square system obtained
by combining our initial  system of SIEs with new constraints derived from
those in~\cite{ganesh2014all}.  We show that this leads to a linear system that enjoys all five properties listed above.
  
There have been (largely unsuccessful) attempts to achieve this goal: 
in  \cite{taskinen2006implementation}, section 6 (see also
\cite{taskinen2006current}), linear constraints were added to ``stabilize''  a system similar to the one in~\cite{ganesh2014all}.
We showed in~\cite{ganesh2014all}
that adding constraints may fail for some particular (but reasonable) values of  the frequency
and dielectric parameters. 
In~\cite{ganesh2014all} we proposed to first multiply the constraint by a
complex parameter $\xi$ and to then add it to the system of linear equations.
We showed in numerical simulations that the use of this parameter can prevent 
the emergence of additional singular cases but we did not provide a formal proof.
 %Here, we have to contend with the hurdle
%that, in the dielectric case,
%a weak constraint had to be imposed in  \cite{ganesh2014all} to guarantee well posedness 
%at all frequencies. It was initially unclear how to impose this constraint without compromising
%the spectral convergence of numerical solutions. 
%This is what we want to resolve
In this present article we describe a new way of incorporating the constraints
using the single layer potential for the Laplace equation.
	This adds to our Maxwell integral equation system 
	a compact operator which is self adjoint and coercive on an adequate subspace.
        
	This article is organized as follows: 
	in Section \ref{pdf+reform} after introducing 
	the Maxwell dielectric  system, we recall
	the system of SIEs	derived in~\cite{ganesh2014all}.  
	We then discuss how several authors (including the co-authors of this paper) have attempted 
	to incorporate linear constraints in the initial linear system in such a way to obtain
	an (unconstrained) well-posed integral equation system.
	Using an explicit example where the system derived in \cite{ganesh2014all}
        becomes singular (previously, only numerical evidence was observed)
        we demonstrate that past attempts have not resolved the issue
        completely.
To close this section, we introduce a novel system of integral equations
obtained  by incorporating constraints
to the  system from~\cite{ganesh2014all} by application of  a compact operator which is
self-adjoint and coercive on an appropriate subspace. 
Our main result is the proof of well-posedness for this novel system,
which is stated in Section \ref{main result}.
The proof relies on introducing a particular function space in which
tangential components of 
the unknown fields are required to be more regular than the normal components.
We prove that the system is in the form identity plus compact. The associated operator 
can be expressed in the form $I + M + \xi J$, where $M$ and $J$ are compact and
$J$ is self-adjoint and positive on an appropriate subspace. We show that this operator 
has a continuous inverse for all sufficiently large $\xi$,
and for all $\xi$ outside a discrete set.
This result is proved using a general functional analysis result
which we state and prove in the appendix,
and the properties of the well-known M\"uller system of integral equations for the 
dielectric problem~c. (M\"uller's system is known to be well-posed, but it becomes unbounded 
at low frequency~\cite{YlaTas:dilow}).

\section{Maxwell dielectric model and  stable reformulations}\label{pdf+reform}
We consider the time-harmonic  electromagnetic wave propagation model problem in three dimensional space comprising a  dielectric body
whose geometry is given by a bounded domain $D \subset \RR^3$ with boundary $\p D$.
Let $D^+ = \RR^3 \setminus \overline{D}$, and $D^- = D$ be the exterior and
interior of $D$ respectively, and set $\epsilon^\pm$ and $\mu^\pm$  to be 
the permittivity and permeability constants 
in   $D^\pm$. The interior  permittivity $\epsilon^-$ has positive real part 
and non-negative imaginary part
while $\epsilon^+$, $\mu^+$, and $\mu^-$ are positive.
 Under this physically appropriate mild assumption on the dielectric medium, 
 the time-harmonic Maxwell partial differential equation (PDE) model is resonance-free~\cite[Theorem 2.1]{costabel2011pap}.
We set
\begin{equation}
\label{eq:epsilon-mu}
\epsilon(\x) = \epsilon^\pm, \qquad \mu(\x) = \mu^\pm, \qquad \mbox{for $\x \in D^\pm$}.
\end{equation}

A time-harmonic incident 
electromagnetic field impinging on the dielectric
body $D$ induces an interior field with spatial components $[\E^-,\H^-]$ in $D^-$ and
a scattered field with spatial components $[\E^+,\H^+]$ in $D^+$.
We set
\begin{equation}\label{eq:int_ext_rep}
\E(\x) = \E^\pm(\x), \qquad \H(\x) = \H^\pm(\x), \qquad \qquad  \x \in D^\pm.
\end{equation}
Then $[\E,\H]$
satisfy the time-harmonic Maxwell equations~\cite[Page~253]{nedelec2001acoustic}
\begin{equation}
\label{eq:reduced-maxwell}
\curl \E(\x) - \i \omega \mu(\x) \H(\x)  =  \zero,  \qquad 
\curl \H (\x) + \i\omega \epsilon(\x) \E(\x)  =  \zero,\qquad  \x \in  D^+,D^-,
\end{equation}
and the Silver-M\"uller radiation condition
\begin{equation}
\label{eq:silver-muller}
\lim_{|\x| \ri \infty} \left[\sqrt{\mu^+}  \H(\x) \times \x -\sqrt{\epsilon^+} | \x | \E(\x) \right] = \zero.
\end{equation}
The  incident electromagnetic
field  $[\E_{\inc},\H_{\inc}]$ is
 required to satisfy the Maxwell equations
\begin{equation}
\label{eq:incident-maxwell}
\curl \E_\inc(\x) - \i \omega \mu^+ \H_\inc(\x)  =  \zero,  \quad 
\curl \H_\inc (\x) + \i\omega \epsilon^+ \E_\inc(\x)  =  \zero,\quad  \x \in  \RR^3 \setminus Q,
\end{equation}
where $Q \subset \RR^3$ is a compact or empty set, bounded away from the 
dielectric body $D$. 
In practice, the dielectric body $D$
is typically subject to 
excitation by an incident plane wave or a point source.
It is convenient to define the total field
\begin{equation}\label{eq:total_field}
\E_\tot(\x) = \E_\tot^\pm(\x), \qquad \H_\tot(\x) = \H_\tot^\pm(\x), \qquad \qquad  \x \in D^\pm,
\end{equation}
which 
is related to the incident field and $[\bE, \bH ]$ by 
\begin{equation}
\E_\tot^- = \E^-, \quad \E_\tot^+ = \E^+ + \E_\inc, \quad
%\end{equation}
%\begin{equation}
\H_\tot^- = \H^-, \quad \H_\tot^+ = \H^+ + \H_\inc.
\end{equation}
We emphasize that inside $D$ the induced field is the total field.

The tangential components of the total electric and magnetic fields
(which are sometimes called the surface currents)
are required to be  continuous across the interface $\surface$, leading
to the interface conditions~\cite[Equation~(5.6.66), Page~234]{nedelec2001acoustic}:
\begin{equation}\label{eq:EH_interface}
 \E^-_\tot \times \n =     \E^+_\tot \times \n,  
\qquad \H^-_\tot \times \n  =    \H^+_\tot \times \n,
\qquad \text{on $\surface$}, 
\end{equation}
where $\n$ is the outward unit normal to $\surface$. 
A consequence of~\eqref{eq:reduced-maxwell},~\eqref{eq:incident-maxwell}, 
and~\eqref{eq:EH_interface}
is that
the normal components of the fields
(which are sometimes called the surface charges) satisfy the interface conditions:
\begin{equation}\label{eq:nor_interface}
\epsilon^- \n \cdot \E_\tot^- =  \epsilon^+ \n \cdot \E_\tot^+, \qquad
\mu^- \n \cdot \H_\tot^- =  \mu^+ \n \cdot \H_\tot^+ \qquad \text{on $\surface$}. 
\end{equation}

\subsection{A reformulated SIE system of the Maxwell model}\label{integral equation ealier}
Denote by
\begin{equation}\label{eq:fund}
G_\pm(\bx,\by)  = \frac{1}{4 \pi} \frac{ e^{\i k_\pm |\bx - \by|} }{|\bx - \by|}
\end{equation}
the free space Green's function for the Helmholtz operator $\Delta + k_\pm^2$, where 
 $k_\pm = \omega \sqrt{\epsilon^\pm \mu^\pm}$ is,  respectively,  the exterior  and
interior wavenumber. 
In \cite{ganesh2014all} we proved that if  the scattered field $\E, \H$ satisfies the Maxwell PDE system 
%(\ref{eq:reduced-maxwell}) and 
%and the Silver-M\"uller radiation condition (\ref{eq:silver-muller}),
(\ref{eq:reduced-maxwell})--(\ref{eq:nor_interface})
then  the exterior traces $\e, \h$  of the total fields on $\p D$ satisfy the system of integral
equations, 
  \begin{align}
   % \allowdisplaybreaks 
(\e \times \n)(\x) 
-  \frac{2}{\epsilon^++\epsilon^-} \Big\{
- \int_{\partial D} \left[\epsilon^+ \frac{\partial G_+}{\partial \n(\x)}
- \epsilon^- \frac{\partial G_-}{\partial \n(\x)}\right]  (\e \times \n)(\y)\; ds(\y) \nonumber  \\
+ \int_{\partial D} \left[\gradx (\epsilon^+ G_+ -\epsilon^- G_-)\right]
 (\e \times \n)(\y)  \cdot [\n(\x) - \n(\y)] \; ds(\y) \nonumber \\
- \i \omega     \int_{\partial D} \left[\epsilon^+ \mu^+ G_+ - \epsilon^- \mu^- G_-\right]
 (\h \times \n)(\y) \times \n(\x) \;ds(\y) \nonumber \\
-  \epsilon^+  \int_{\partial D} 
\left[\gradx (G_+ - G_-) \times \n(\x)\right]   (\e \cdot \n)(\y)\;    ds(\y) \Big \} %\nonumber  \\
= \frac{2 \epsilon^+}{\epsilon^++\epsilon^-}  (\e_{\inc} \times \n)(\x),
\label{eq:eq_1} \\
(\e \cdot \n)(\x) 
-   \frac{2 \epsilon^-}{\epsilon^++\epsilon^-}  \Big \{ - \i \omega  \int_{\partial D} \left[\mu^+G_+-\mu^-G_-\right] 
(\h\times \n)(\y)  \cdot \n(\x)\; ds(\y) \nonumber \\ - 
 \int_{\partial D} \left[\frac{ \partial G_+}{\partial \n(\x)} - \frac{\epsilon^+}{\epsilon^-} \frac{ \partial G_-	}{\partial \n(\x)}\right]
(\e\cdot \n)(\y) \;ds(\y) \nonumber \\
+ \int_{\partial D}   \left[\gradx (G_+ -G_-) \times \n(\x)\right] \cdot (\e \times \n)(\y) \; ds(\y) \nonumber \\
=\frac{2 \epsilon^-}{\epsilon^++\epsilon^-} (\e_{\inc} \cdot \n)(\x), \label{eq:eq_2} \\
(\h \times \n)(\x)  
- \frac{2}{\mu^+ + \mu^-} 
\Big \{   - \int_{\partial D} \left[\mu^+ \frac{\partial G_+}{\partial \n(\x)} 
- \mu^- \frac{\partial G_-}{\partial \n(\x)}\right ] (\h \times \n)(\y)\; ds(\y) \nonumber  \\
+ \int_{\partial D} \left[\gradx (\mu^+ G_+ -\mu^- G_-)\right]
 (\h \times \n)(\y)) \cdot [\n(\x) - \n(\y)] \; ds(\y) \nonumber \\
+ \i \omega     \int_{\partial D} \left[\epsilon^+ \mu^+ G_+ - \epsilon^- \mu^- G_-\right]
 (\e \times \n)(\y)\times \n(\x) \; ds(\y) \nonumber \\
-   \mu^+  \int_{\partial D} 
\left[\gradx (G_+ - G_-) \times \n(\x)\right]  (\h \cdot \n)(\y)\;    ds(\y) \Big \} % \nonumber  \\
=  \frac{2 \mu^+}{\mu^+ + \mu^-}(\h_{\inc }  \times \n)(\x),  \label{eq:eq_3} \\
 (\h \cdot \n)(\x)   
- \frac{2 \mu^-}{\mu^+ + \mu^-} \Big \{
 \i \omega  \int_{\partial D} \left[\epsilon^+G_+-\epsilon^-G_-\right] (e\times \n)(\y)  \cdot \n(\x) \; ds(\y) 
 \nonumber \\
 - \int_{\partial D} \left[\frac{ \partial G_+}{\partial \n(\x)} 
- \frac{\mu^+}{\mu^-} \frac{ \partial G_-}{\partial \n(\x)}\right]
(\h\cdot \n)(\y) \;ds(\y) \nonumber \\
+ \int_{\partial D}   \left[\gradx (G_+ -G_-) \times \n(\x)\right] \cdot (\h \times \n)(\y) \; ds(\y) 
=  \frac{2 \mu^-}{\mu^+ + \mu^-}  \h_{\inc }\cdot \n(\x),  \label{eq:eq_4}
\end{align}
where we use the abbreviation $G_\pm$  for $G_\pm(\x,\y)$, and 
$(\e_{\inc }, \h_{\inc })$ is the trace of the incident field $(\E_{\inc}, \H_{\inc})$ 
on $\p D$.
It is convenient to write this system in operator form. Setting
\begin{eqnarray}
\e^i   &=& \frac{2\epsilon^+}{\epsilon^+ + \epsilon^-} \n\times  (\e_{\inc } \times \n) + 
\frac{2 \epsilon^-}{\epsilon^+ + \epsilon^-} \n ( \e_{\inc }\cdot \n),  \label{eq:einp}  \\
\h^\i &=& \frac{2\mu^+}{\mu^+ + \mu^-} \n \times  (\h_{\inc } \times \n) + 
\frac{2 \mu^-}{\mu^+ + \mu^-} \n  (\h_{\inc }   \cdot \n), \label{eq:hinp} 
\end{eqnarray} 
the system (\ref{eq:eq_1})--(\ref{eq:eq_4}) is equivalent to 
\bean \label{Msystem}
(\Ima + \Mm ) (\e, \h) = (\e^i, \h^i),
\eean
where $\Ima$ is the identity operator and the explicit form of $\Mm$ can be found in~\cite[Equation~(3.17)]{ganesh2014all} or directly inferred from (\ref{eq:eq_1})--(\ref{eq:eq_4}).

Using the assumption that $\p D$ is smooth, it was explained in  \cite{ganesh2014all} that for any $s \in \RR$,
the operator $\Mm$ is continuous from $H^s(\p D)^3 \times H^s(\p D)^3 $ to 
$H^{s+1}(\p D)^3 \times H^{s+1}(\p D)^3 $. Consequently, the system (\ref{Msystem}) is of the form ``identity plus compact''.
For the mathematical analysis and the operator properties used in this article, it will be convenient to  
assume throughout the article that the  dielectric body $D$ is smooth. 
However, as  in~\cite{ganesh2014all}, most   results in this article will still hold  when $D$ is a $C^{1, \alpha}$ domain.

Many authors have   derived systems of integral equations for the free space dielectric 
problem in this  ``identity plus compact'' form. However they usually suffer from low-frequency 
break down, meaning that the norm of the inverse operator blows up
as the frequency tends to zero~\cite{YlaTas:dilow}.
One of the most interesting contributions of our system 
(\ref{Msystem}) is that $\Ima + \Mm $ is norm 
convergent to an invertible operator as the frequency $\omega$ tends to zero (see \cite[Theorem~B.1]{ganesh2014all}).
However, the operator   $\Ima + \Mm $ may be singular
for some frequencies, as demonstrated in numerical simulations in~\cite{ganesh2014all} and demonstrated analytically in section \ref{explicit example}.

\subsection{Constraints for all-frequency stabilization of the SIE (\ref{Msystem})}
\label{sub sec constraints}
Although the system
(\ref{Msystem})  is always solvable, the solution may not be unique for some frequencies because
$\Ima + \Mm $ may have a nontrivial kernel.
However, under the constraints
  \bean
	\divs(\e \times \n)=i \omega \mu^+ \h \cdot \n, \quad
\divs (\h \times \n) = - i \omega \epsilon^+ \e \cdot  \n,
\label{bad constraint}
\eean
 the solution to (\ref{Msystem})  is unique. This has been known for some time, see \cite[chapter~VI]{muller:book}.
%*********** add reference to break down ***************??
In this work we 
want to find  a constraint in integral equation form, rather than PDE form.
To that end, in~\cite{ganesh2014all} we defined
operators
\begin{eqnarray}
(\SSS_\pm w)(\x) =  \int_{\surface} G_\pm(\x,\y)  \; w(\y) \; ds(\y),  \quad \x \in  \surface, \label{eq:S} \\
(\SK_\pm \w)(\x)  =  \int_{\surface}  \Big[ \gradx G_\pm(\x,\y) \times \n(\y) \Big] \cdot \w(\y) \; ds(\y), 
\quad \x \in \surface,
\label{eq:K} 
\end{eqnarray}
for a scalar function $w$ and a vector field $\bw$ on $\p D$,
and
\bean
J_1 (\e , \h) (\x ) = 
(-\i \omega \epsilon^+ (\SSS_+ - \SSS_-)(\e \cdot \n) + (\SK_+ - \SK_-) \h ), \label{J1} \\
J_2 (\e , \h) (\x ) =  (\i \omega \mu^+ (\SSS_+ - \SSS_-)(\h \cdot \n) + (\SK_+ - \SK_-) \e ) \label{J2}.
\eean
We proved in \cite[Theorem~5.4]{ganesh2014all} that the augmented system
\bean
\label{augmented}
(\Ima + \Mm ) (\e, \h) = (\e^i, \h^i), \q  J_1 (\e , \h)=J_2 (\e , \h)=0,
\eean
has a unique solution
for all frequencies $\omega >0$ and  values of the electromagnetic parameters $\epsilon, \mu$ considered in this paper.
%(complex values of )

\subsection{Stabilized SIE  derived from the augmented system (\ref{augmented})}
\label{subsec derived systems}
In applications, linear systems
derived from formulation (\ref{augmented}) by discretizing may be very large,
and they may have to be solved a large number of times for different values of the frequency 
or of the electromagnetic parameters.
Consequently, imposing the two constraints from (\ref{augmented}) 
may become expensive from a computational perspective for  high frequencies and
complex geometries. 

There  have been several attempts to find a system of linear integral equations 
which is equivalent to the augmented system (\ref{augmented}). 
 In \cite[Section~6]{taskinen2006implementation} (see also
\cite{taskinen2006current}), adding (arithmetically) conditions (\ref{J1})--(\ref{J2}) to 
(\ref{eq:eq_2}) and (\ref{eq:eq_4})
% in case if system 
%(\ref{eq:eq_1}-\ref{eq:eq_4}), which worked on a few examples.
was shown to be effective in a few examples.
Numerical 
evidence  in~\cite{ganesh2014all} demonstrated that this does not work for all frequencies, even in the simple case 
where $\p D$ is a sphere.
In \cite[Section~7]{ganesh2014all} we demonstrated that adding
a multiple of conditions  (\ref{J1})--(\ref{J2})
to the system (\ref{eq:eq_1})--(\ref{eq:eq_4})
was effective for all frequencies (with the multiple depending on the
frequency).

It is fruitful to recast this idea in terms of general operator theory. Let $H$ be a separable complex Hilbert space and $M$
and $J$ two linear compact operators from $H$ to $H$. Suppose that we know that for a given $b \in H$ the constrained linear equation 
\bean 
(I + M) x= b, \q J x = 0
\label{constrained}
\eean 
is uniquely solvable. It follows that $(I + M + \xi J) x= b $
for any $\xi$ in $\CC$.
Can we tell with certainty that for some $\xi$ in $\CC$ the operator $I + M + \xi J$ is invertible
with bounded inverse?

If $H$ is finite-dimensional, this relates to the so called theory of pencils, \cite{gantmacher2005applications}.
If $H$ is finite- dimensional and $J$ is regular, we write $I + M + \xi J = J (J^{-1} (I+M) + \xi I)$, which is regular, except for finitely many values of 
$\xi$. If $J$ is singular and $H$ is finite- dimensional then, using the determinant, we see that $I + M + \xi J $ is either regular, except for finitely many values of  $\xi$, or it is singular for all values of $\xi$. In the latter case  $I + M + \xi J $ is called a singular pencil.

Beside the case 
where $H$ is finite- dimensional and $J$ is regular,
are there  simple to use conditions on $M$ and $J$ that will guarantee that $I + M + \xi J $ is not a singular pencil?
We note that it is not sufficient that
the intersection of the nullspaces of $I + M$ and $J$ is  trivial.
Here is an example in dimension 3, which can easily be generalized to any higher (including infinite)
dimension. Set 
\begin{equation}
I+ M= \left(
\begin{array}{ccc}
0 & 0 & 1 \\
0 & 0  & 1  \\
1 & 1 & 0
\end{array}
\right), \quad 
J= \left(
\begin{array}{ccc}
1 & 0 & 0 \\
0 & -1 & 0 \\
0 & 0 & 0
\end{array}
\right) \nonumber
\end{equation}
and let $e_1, e_2, e_3$ be the natural basis of $\CC^3$.
Then $e_3$ spans the nullspace of $J$ and is  
not in the nullspace of $I +M$.  Clearly, $ (I + M + \xi J)(e_1 -  e_2 - \xi e_3) =0 $, for all $\xi \in \CC$. 

We prove in Theorem~\ref{all symmetric} in the appendix that  a sufficient condition for the pencil 
$I + M + \xi J$ to be regular is to satisfy (i) $M$ and $J$ are compact;
(ii)  the intersection of the nullspaces of $I + M$ and $J$ is trivial;
(iii)  $J$ is self-adjoint and  non-negative; and (iv) $N(J)$ is invariant under $M$. 
 %(this result is even valid in infinite dimension). 
These strong conditions on $J$ and on $M$ are required.
In Proposition \ref{injective counter example} in the appendix we provide
an example, valid in any separable Hilbert space, of two compact operators $J$ and $M$ 
such that $J$ is injective and yet  $I + M + \xi J$ has a non-trivial nullspace, for all $\xi \in \CC$.
The following proposition 
explains the  the numerical results in  \cite[Section~7]{ganesh2014all},
particularly the isolated peak in \cite[Figure~7.2]{ganesh2014all}. 

\begin{prop}\label{prop discrete}
Let $H$ be a separable Hilbert space and   $M$ and $J$ two compact linear operators from $H$ to $H$.
Then $ I + M + \xi J$ is either singular for all values $\xi  \in \CC$ or there is a discrete set $Z$ such that 
$ I + M + \xi J$  is invertible for $\xi \in \CC \setminus Z$.
\end{prop}

\begin{proof}
Supose there is a $\xi_0  \in \CC$ such that
 $ I + M + \xi_0 J$  is invertible. Write
\bea
I + M + \xi J = (I + M + \xi_0 J) ( I +  (\xi - \xi_0 ) (I + M + \xi_0 J)^{-1} J ).
\eea
We then notice  that $(I + M + \xi_0 J)^{-1} J$ is a compact operator, so according to 
basic functional analysis theory, $I +  (\xi - \xi_0 ) (I + M + \xi_0 J)^{-1} J$ is invertible except
possibly for a $\xi - \xi_0 $ in a discrete subset of $\CC$.
\end{proof}

\subsection{An explicit sphere case  where equations (\ref{eq:eq_1}-\ref{eq:eq_4}) are singular}
\label{explicit example}
To find a simple case where  equations (\ref{eq:eq_1})--(\ref{eq:eq_4}) are singular, we 
set $\e \times \n$ and $\h \times \n$ to be  zero in
(\ref{eq:eq_1})--(\ref{eq:eq_4})
to obtain
 the reduced homogeneous  system  
 \begin{align}
   \allowdisplaybreaks
  \int_{\partial D} 
\left[\gradx (G_+ - G_-) \times \n(\x)\right]   (\e \cdot \n)(\y)\;    ds(\y) %\nonumber  \\
= \zero,
\label{reduced1} \\
(\e \cdot \n)(\x) 
+  \frac{2 \epsilon^-}{\epsilon^++\epsilon^-}  
 \int_{\partial D} \left[\frac{ \partial G_+}{\partial \n(\x)} - \frac{\epsilon^+}{\epsilon^-} \frac{ \partial G_-	}{\partial \n(\x)}\right]
(\e\cdot \n)(\y) \;ds(\y) 
=0, \label{reduced2} \\
\int_{\partial D} 
\left[\gradx (G_+ - G_-) \times \n(\x)\right]  (\h \cdot \n)(\y)\;    ds(\y) % \nonumber  \\
=  \zero  \label{reduced3} \\
 (\h \cdot \n)(\x)   
+\frac{2 \mu^-}{\mu^+ + \mu^-} 
  \int_{\partial D} \left[\frac{ \partial G_+}{\partial \n(\x)} 
- \frac{\mu^+}{\mu^-} \frac{ \partial G_-}{\partial \n(\x)}\right]
(\h\cdot \n)(\y) \;ds(\y) 
= 0.  \label{reduced4}
\end{align}
%Cite: Well-Conditioned Müller Formulation for
%Electromagnetic Scattering by
%Dielectric Objects
%Pasi Ylä-Oijala and Matti Taskinen
%They show low frequency breakdown - not that their T mueller is an equation of the first kind.\\
Our example uses the following technical lemma.
\begin{lem}\label{sphere}
For any complex number $k$, and $\x \in \RR^3$,
let $v(\x) =\ds \int_S \f{1}{4 \pi} \f{e^{i k |\x-\y|}}{|\x-\y|} d s(\y)$,
where $S$ denotes the unit sphere in $\RR^3$ centered at the origin. Let $r =|\x|$.
Then
$$
v(\x)
 = \left\{  
\begin{array}{lll} \ds
\f{e^{ikr}}{r} \f{\sin k }{k},& \qquad & \mbox{ if } r \geq 1, \\
\ds e^{ik} \f{\sin k }{k},&& \mbox{ if } r = 1, \\
\ds \f{\sin kr }{r} \f{e^{ik}}{k},&& \mbox{ if } r \leq 1.
\end{array}
\right.
$$
\end{lem}

\begin{proof}
Let $R$ be a rotation of $\RR^3$ with axis passing through the origin.
Then
\bea
v(R \x ) = \ds \int_S \f{1}{4 \pi} \f{e^{i k |R\x-\y|}}{|R\x-\y|} d s(\y) 
= \int_S \f{1}{4 \pi} \f{e^{i k |\x- R^T\y|}}{|\x- R^T\y|} d s(\y) = v(\x).
\eea
% long explanation:
%\vskip .1in
%one way of parametrizing the sphere
%is $\y = (\cos  \theta \sin \phi, \sin \theta \sin \phi , \cos \phi)$ \\
%another is
%$\y = (\cos  (\theta -\theta ') \sin (\phi - \phi '), \sin (\theta -\theta ')  \sin (\phi - \phi ') , \cos (\phi - \phi '))$, \\
%with $0 \leq \theta < 2 \pi$, $ \phi ' \leq \phi < \phi' + \pi $, with surface element $ \sin (\phi - \phi ') $
%\vskip .1in
Thus $v $ is a radial function and we can write $v(\bx) = \tilde{v} (r)$. Away from $S$ and the origin, $v$ satisfies the differential equation
%%$r^{-2} \p_r (r^2 \p_r v) + k^2 v =0$.  
$ (\Delta + k^2) v =0$, thus $\tilde{v}$ satisfies for $r$ in $(0,1)$ and $r$ in $(1,\infty)$, 
\bea 
\p_r^2 \tilde{v} + 2 r^{-1} \p_r \tilde{v}+ k^2 v =0 . %\label{ode}
\eea
If we set $f(r)=r\tilde{v} (r)$, $f$ satisfies  
$\p_r^2 f + k^2 f=0$, thus $f(r)$ must be a linear combination of $e^{ikr}$ and $e^{-ikr}$.
It follows that $\tilde{v} (r)$ is  a linear combination of $\f{e^{ikr}}{r}$ 
and $\f{e^{-ikr}}{r}$ in $(0,1)$ and in $(1,\infty)$.
Since $\tilde{v}$ is continuous at zero, $\tilde{v}(r)$ must be a multiple of $\f{\sin kr }{r}$ for
$r$ in $[0,1]$. 
From the definition of $v$ we have  $\tilde{v}(0) = v((0,0,0))= e^{ik}$,  
%\bea
%\tilde{v}(1)  = v((0,0,1)) = \f12 \int_0^\pi \f{e^{2i k \sin \f{\phi}{2}}}{\sin \f{\phi}{2} } \sin \phi \, d \phi
%= \int_0^1 e^{2ik t } dt =  e^{ik} \f{\sin k }{k},
%\eea
thus  $\tilde{v}(r) = \ds \f{\sin kr }{r} \f{e^{ik}}{k} , \mbox{ if }  0 \leq r  \leq  1$. 
%Next, for $r$ in $(1, \infty)$, $v(r)$ still satisfies the ordinary differential equation (\ref{ode}).

Next we use the fact that $v$ is continuous across $S$ and that the normal derivative of $v$
satisfies the jump condition $(\p_r v)^+ - (\p_r v)^- =-1$, to find the expression of $v$ as function
of $r$ outside $S$.
\end{proof}

We now use Lemma~\ref{sphere} to construct a solution
of~(\ref{reduced1})--(\ref{reduced4}).
First we define
$$
w(\x) = \left\{  
\begin{array}{lll} \ds
\f{e^{ik_- r}}{r} \f{\sin k_-  }{k_-},& \qquad & \mbox{ if } r >  1, \\
\ds \f{\sin k_+ r }{r} \f{e^{ik_+}}{k_+}, && \mbox{ if } r <  1.
\end{array}
\right.
$$
Across $S$ we see that
\bean \label{eps k eq}
 \epsilon^+(\partial_r w)^+ - \epsilon^- (\partial_r w)^-
= \epsilon^+ \sin  k_-   e^{i k_-}  \left(i- \f{1}{k_-} \right)
- \epsilon^-  e^{i k_+ } \left( \cos k_+  - \f{\sin k_+ }{k_+ }\right).
\eean 
There are infinitely many values of  $\epsilon^+, \epsilon^-, k_+, k_- $
for which
the right hand side of~(\ref{eps k eq}) is zero.
These values can be found numerically.
For example, when
$\epsilon^+=1, \epsilon^-=6 $
one can find a zero of~(\ref{eps k eq}) at
\bea
k_+ = 0.763 452 368 18~\mbox{(to 11 d.p.)}%\dots,%02931614946743781407420101826369629931697305, \\ 
\quad k_- = 1.835 368 158 62~\mbox{(to 11 d.p.)} %\dots %8437141232745158405948189411893514882494048
\eea
For such zeros of~(\ref{eps k eq}), if $\e \cdot \n =1, \, \h \cdot \n =0$
then (\ref{reduced3}) and (\ref{reduced4})  are trivially satisfied. 
Using Lemma~\ref{sphere}  we obtain
$$
w(\x) = \left\{  
\begin{array}{lll} \ds
\int_S \f{1}{4 \pi} \f{e^{i k_- |\x-\y|}}{|\x-\y|} ds(\y), & \qquad & \mbox{ if } r >  1, \\
\ds \int_S \f{1}{4 \pi} \f{e^{i k_+ |\x-\y|}}{|\x-\y|} ds(\y), & & \mbox{ if } r <  1.
\end{array}
\right.
$$
Consequently, using~(\ref{eps k eq}) 
\bean
1
+  \frac{2 \epsilon^-}{\epsilon^++\epsilon^-}  
 \int_{S} \left[\frac{ \partial G_+}{\partial \n(\x)} - \frac{\epsilon^+}{\epsilon^-} \frac{ \partial G_-	}{\partial \n(\x)}\right]
 \;ds(\y) \no \\
= -  \frac{2 \epsilon^-}{\epsilon^++\epsilon^-}   (\epsilon^+ \sin  k_-   e^{i k_-}  (i- \f{1}{k_-} )
- \epsilon^-  e^{i k_+ } ( \cos k_+  - \f{\sin k_+ }{k_+ }) ),  \label{eps k eq2}
\eean
and hence (\ref{reduced2}) is  satisfied for the above specific choice of dielectric parameters and 
$\e \cdot \n =1$. Further,  for $\x \in S$,
\begin{displaymath}
  \int_S (G_+ - G_-) \, 1 \, ds(\y)
  = (w)^- - (w)^+,
\end{displaymath}
which is constant  due to Lemma \ref{sphere},  so that its surface gradient is zero
on $S$, and thus (\ref{reduced1}) is also satisfied.

In conclusion if $\p D = S$
we have found that 
 the system (\ref{eq:eq_1}-\ref{eq:eq_4}) is singular for these values of 
$\epsilon^+, \epsilon^-, k_+, k_- $.

\subsection{A new way of incorporating  constraints (\ref{bad constraint}) into  (\ref{eq:eq_1}--\ref{eq:eq_4})}
\label{new way}
In order to established a well-posed SIE system for all frequencies, we consider a new approach consisting
of  incorporating  the constraints (\ref{bad constraint}) in surface integral form. 
We recall
the single layer potential for the Laplacian
\bean
\bS w(\x) = \f{1}{4 \pi }\int_{\p D} \f{1}{|\x - \y|} w(\y) \, ds(\y)
\label{Sdef}
 \eean 
and the associated gradient vector potential
\bean 
\bD \w(\x) = \f{1}{4 \pi }\int_{\surface}  \left[ \gradx \left(\f{1}{ |\x - \y|} \right) \times \n(\y) \right] \cdot \w(\y) \; ds(\y).
\label{Ddef}
\eean
The operator $\bS$ has the following
properties~\cite{ammarikangbook, nedelec2001acoustic}.
\begin{prop} \label{prop S}
  For any $s \in \RR$, the operator $\bS: H^s(\p D) \to H^{s+1}(\p D)$ is continuous  and invertible with continuous inverse.
  Furthermore, the operator
  $\bS: L^2 (\p D) \to L^2 (\p D)$ is  self  adjoint and coercive.
\end{prop}

\begin{lem}
Let $\e, \h \in L^2(\p D)^3$.
The constraints~(\ref{bad constraint}) are equivalent to 
\bean
(-\i \omega \epsilon^+ \bS(\e \cdot \n) + \bD \h) & = & 0, \label{new cs 1} \\
  (\i \omega \mu^+\bS(\h \cdot \n) + \bD  \e ) & = & 0.\label{new cs 2}
\eean
\end{lem}

\begin{proof}
Applying Green's theorem on the closed
surface $\p D$ establishes that~(\ref{new cs 1})--(\ref{new cs 2})
are equivalent to
\begin{align*}
  \begin{split}
\bS (\divs(\e \times \n)- i \omega \mu^+ \h \cdot \n) & = 0,\\
\bS (\divs (\h \times \n) + i \omega \epsilon^+ \e \cdot  \n) & =0.
  \end{split}
\end{align*}
The result then follows 
by Proposition \ref{prop S}.
\end{proof}

We now define the operator $\Jm$ by setting
\bean
\Jm(\e , \h) (\x ) = 
\Big( \n(\x)   (\omega^2 \epsilon^+ \bS(\e \cdot \n) + i \omega \bD \h),  \n(\x)
%\phantom{\Jm(\e , \h) (\x ) =}  
 ( \omega^2 \mu^+\bS(\h \cdot \n) - i \omega \bD  \e )
\Big). \label{defJ}
\eean
The constrained problem (\ref{augmented}) implies that for any
$\xi \in \CC$,
\bean
\label{new eq}
(\Ima + \Mm + \xi \Jm) (\e, \h) = (\e^i, \h^i).
\eean
This formulation does not suffer from 
low-frequency breakdown. Indeed,
$\Ima + \Mm $ converges in   operator norm 
to an invertible linear operator as the frequency $\omega \to 0$
~\cite[Appendix~B]{ganesh2014all}
and clearly $\Jm$ converges in   operator norm to zero as
$\omega \to 0$.

%\subsection{New functional space}
\section{Unique solvability of The SIE system  (\ref{new eq})} \label{main result}
To prove that equation (\ref{new eq}) is well-posed (for a range of $\xi$ that we specify later) 
we introduce the space
\bea
 X =\{ (\e ,\h): \e\times \n,  \h \times \n \, \in H^1(\p D)^3, \  \e\cdot \n,  \h \cdot \n \, \in L^2(\p D)
\}
\eea
with norm $\| \, \|_X$ defined by
\bea
\| (\e,\h) \|^2_X = 
  \|\e\times \n\|_{H^1(\p D)^3}^2 +  \|\h \times \n \|_{H^1(\p D)^3}^2 +   \|\e\cdot \n
	\|_{L^2(\p D)}^2 + \|\h \cdot \n\|_{L^2(\p D)}^2.
\eea
and  associated induced inner product from each term. 
It is clear that $X$ is a Hilbert space.

\begin{prop} \label{M is compact}
The linear operator $\Mm: X \to X$ is compact.
\end{prop}

\begin{proof}
Recall that $\Mm$ was defined through the system (\ref{eq:eq_1})--(\ref{eq:eq_4}).
We first examine the terms from  (\ref{eq:eq_1}),
\bea
 \int_{\partial D} \left[\epsilon^+ \frac{\partial G_+}{\partial \n(\x)}
- \epsilon^- \frac{\partial G_-}{\partial \n(\x)}\right]  (\e \times \n)(\y)\; ds(\y), \nonumber  \\
 \int_{\partial D} \left[\gradx (\epsilon^+ G_+ -\epsilon^- G_-)\right]
 (\e \times \n)(\y)  \cdot [\n(\x) - \n(\y)] \; ds(\y), \nonumber \\
 \i \omega     \int_{\partial D} \left[\epsilon^+ \mu^+ G_+ - \epsilon^- \mu^- G_-\right]
 (\h \times \n)(\y) \times \n(\x) ds(\y) .
\eea
It is clear from integral operator theory~\cite{colton:integral, coltoninverse, nedlec:book} that if 
$(\e,\h) \in X$ then
each of these three terms is in 
$H^2(\p D)^3$,  and the associated linear operators are bounded. That is, the terms
depend continuously on the $H^1$ norms of 
  $\e \times \n$ and $\h \times \n$. 
Next we examine the term
\bean \label{term}
\int_{\partial D} 
\left[\gradx (G_+ - G_-) \times \n(\x)\right]   (\e \cdot \n)(\y)\;    ds(\y)
\eean
in  (\ref{eq:eq_1}).
Using the Taylor series of the exponential,
\begin{displaymath}
\gradx (G_+ - G_-) = \left( \frac{ (i  k_-)^2 
- (i  k_+)^2 }{8 \pi} + O( | \bx - \by | ) \right) \gradx  | \bx - \by |,
\end{displaymath}
from integral operator theory~\cite{colton:integral, coltoninverse, nedlec:book} 
the term (\ref{term}) is in $H^2(\p D)^3 $ if $(\e,\h) \in X$ and
depends continuously on the $L^2$ norm of 
  $\e \cdot \n$.

We now write the terms  in (\ref{eq:eq_2}) involved in defining $\Mm$,
\bea
\int_{\partial D} \left[\mu^+G_+-\mu^-G_-\right] 
(\h\times \n)(\y)  \cdot \n(\x)\; ds(\y) ,\\ 
 \int_{\partial D} \left[\frac{ \partial G_+}{\partial \n(\x)} - \frac{\epsilon^+}{\epsilon^-} \frac{ \partial G_-	}{\partial \n(\x)}\right]
(\e\cdot \n)(\y) \;ds(\y), \nonumber \\
 \int_{\partial D}   \left[\gradx (G_+ -G_-) \times \n(\x)\right] \cdot (\e \times \n)(\y) \; ds(\y).\nonumber 
\eea
It is clear that each of these terms is in $H^1(\p D)$ and that they
depend continuously
on the $H^1(\p D)^3 $ norms of $\e \times \n$ and $\h \times \n$ and 
 on the $L^2$ norm of 
 $\e \cdot \n$.
 
The analysis of the terms in ~(\ref{eq:eq_3}) and~(\ref{eq:eq_4}) is similar.
 \end{proof}

\begin{prop} \label{J is compact}
The linear operator $\Jm: X \to X$ is compact.
\end{prop}

\begin{proof}
We see from definition~(\ref{Ddef}) that $\bD \w$ depends 
on $\w$ only through its tangential part. 
Thus since $\h \times \n \in H^1(\p D)^3$, 
$\bD \h \in H^1(\p D)$ and depends continuously 
on the $H^1(\p D)^3$
norm of $\h \times \n$.
$\bS(\e \cdot \n) $ is also in $H^1(\p D)$ and it depends continuously on the $L^2(\p D)$
norm of $\e \cdot \n$.

A similar argument applies to the second component of $\Jm$.
Thus $\Jm: X \to X$ is compact.
\end{proof}

Let us  now write a reduced system in $\e \cdot \n,  \, \h \cdot \n$  derived 
from (\ref{new eq}) where we set $\e \times \n = \h \times \n =0 $,
\bea
(\e \cdot \n)(\x) 
+  \frac{2 \epsilon^-}{\epsilon^++\epsilon^-}  
 \int_{\partial D} \left[\frac{ \partial G_+}{\partial \n(\x)}  - \frac{\epsilon^+}{\epsilon^-} \frac{ \partial G_-	}{\partial \n(\x)}\right]
(\e\cdot \n)(\y) \;ds(\y) 
+ \xi \omega^2 \epsilon^+\bS (\e \cdot \n)
=f, \label{new reduced1} \\
(\h \cdot \n)(\x)   
+\frac{2 \mu^-}{\mu^+ + \mu^-} 
  \int_{\partial D} \left[\frac{ \partial G_+}{\partial \n(\x)} 
- \frac{\mu^+}{\mu^-} \frac{ \partial G_-}{\partial \n(\x)}\right]
(\h\cdot \n)(\y) \;ds(\y) +  \xi \omega^2 \mu^+ \bS (\h \cdot \n)
= g  \label{new reduced2},
\eea
for two given forcing terms $f$ and $g$.
It is helpful to write this system in operator form as
\bean
(\e \cdot \n ) + \bK_1 (\e \cdot \n ) + \xi \omega^2 \epsilon^+\bS (\e \cdot \n) = f, 
\label{opform e dot n}\\
(\h \cdot \n ) + \bK_2 (\h \cdot \n ) + \xi \omega^2 \mu^+\bS (\h \cdot \n) = g,
\label{opform h dot n}
\eean
where
\begin{align*}
  \bK_1 w(\x) &  = \frac{2 \epsilon^-}{\epsilon^++\epsilon^-}  
 \int_{\partial D} \left[\frac{ \partial G_+}{\partial \n(\x)}  - \frac{\epsilon^+}{\epsilon^-} \frac{ \partial G_-	}{\partial \n(\x)}\right]
 w(\y) \;ds(\y),\\
 \bK_2 w(\x) & = \frac{2 \mu^-}{\mu^+ + \mu^-} 
  \int_{\partial D} \left[\frac{ \partial G_+}{\partial \n(\x)} 
- \frac{\mu^+}{\mu^-} \frac{ \partial G_-}{\partial \n(\x)}\right]
w(\y) \;ds(\y).
\end{align*}

\begin{prop}
  There is a $\xi_0 > 0$ such that for all $\xi > \xi_0$,
  equations~(\ref{opform e dot n}) and~(\ref{opform h dot n})
  are well posed for $f, g \in L^2(\p D)$.
  In addition, there is a constant $C$ such that for all $\xi > \xi_0$,
\bea
\| \e \cdot \n \|_{L^2{(\p D)}} \leq C \| f \|_{L^2{(\p D})} , \q
\| \h \cdot \n \|_{L^2{(\p D)}}\leq C \| g \|_{L^2{(\p D)}} .
\eea
\end{prop}
\begin{proof}
The operator $\bK_1 :L^2{(\p D)} \to L^2{(\p D)}$ is compact\cite{colton:integral, coltoninverse, nedlec:book}.
Since $\bS$ satisfies the properties stated in Proposition  \ref{prop S},  we can apply Theorem \ref{norm inv lemma}.
A similar argument establishes well posedness of~(\ref{opform h dot n})
and the second
inequality.
\end{proof}
%The operator $\bK_1 :L^2{(\p D)} \to L^2{(\p D)}$ is compact\cite{colton:integral, coltoninverse, nedlec:book}.
%Since $\bS$ satisfies the properties stated in Proposition  \ref{prop S},  we can apply Theorem \ref{norm inv lemma}.
%A similar argument establishes well posedness of~(\ref{opform h dot n})
%and the second
%inequality.
%\end{proof}

%\begin{prop}
%The linear operator $\Ima + \Mm + \xi \Jm$  is norm convergent
%to a ..... as $\omega$ tends to zero
%\end{prop}
% already mentioned

\begin{thm} \label{maintheorem}
The linear operator $\Ima + \Mm + \xi \Jm: X \to X$ 
is invertible for all $\xi > \xi_0$, and also
for all $\xi \in \CC \setminus Y$ where $Y$ is a discrete
   set.
The corresponding inverse operator is uniformly bounded for all $\xi>\xi_0$.
%Plus all over $\xi$'s, additional regularity ...
\end{thm}

\begin{proof}
To prove that $\Ima + \Mm + \xi \Jm$  is invertible for all $\xi$ sufficiently large, 
since $\Mm $ and $\Jm$ are compact operators,
we can argue by contradiction by 
assuming that there is a sequence $(\xi_p) \in \RR$
with $\xi_p \to \infty$ as $p \to \infty$,
and a sequence 
$(\e_p, \h_p) \in X$ such that
\bean
\|\e_p\times \n\|_{H^1(\p D)^3}^2 +  \|\h_p \times \n \|_{H^1(\p D)^3}^2 +   \|\e_p \cdot \n
	\|_{L^2(\p D)}^2 + \|\h_p \cdot \n\|_{L^2(\p D)}^2 =1, \label{norm1}
\eean
and
 \bean 
(\Ima + \Mm + \xi_p\Jm) (\e_p, \h_p) =0. \label{p eq}
 \eean
We first write the equation for $(\e_p, \h_p)$ corresponding to (\ref{p eq}) in the first normal component,
\bean
(\e _p\cdot \n)(\x) 
-   \frac{2 \epsilon^-}{\epsilon^++\epsilon^-}  \Big \{ - \i \omega  \int_{\partial D} \left[\mu^+G_+-\mu^-G_-\right] 
(\h_p \times \n)(\y)  \cdot \n(\x)\; ds(\y) \nonumber \\ - 
 \int_{\partial D} \left[\frac{ \partial G_+}{\partial \n(\x)} - \frac{\epsilon^+}{\epsilon^-} \frac{ \partial G_-	}{\partial \n(\x)}\right]
(\e_p\cdot \n)(\y) \;ds(\y) \nonumber \\
+ \int_{\partial D}   \left[\gradx (G_+ -G_-) \times \n(\x)\right] \cdot (\e_p \times \n)(\y) \; ds(\y) \nonumber \Big\}\\
+ \xi_p \omega \bS (\omega \epsilon^+ \e_p \cdot \bn -i \divs (\h_p \times \n)  )
=0, \label{hom eq in e dot n}
\eean
where Green's theorem was used to introduce the term $\divs (\h_p \times \n)$.
It follows from Proposition \ref{prop S} that
\bean
 \|   \divs (\h_p \times \n) +  i \omega \epsilon^+ \e_p \cdot  \n
 \|_{H^{-1}(\p D)} = O( \xi_p^{-1}) . \label{max cond is small}
\eean 
We now rearrange terms in the row of equation (\ref{p eq}) corresponding in the first tangential 
components to obtain
\begin{align}
(\e \times \n)(\x) 
-  \frac{2}{\epsilon^++\epsilon^-} \Big\{
- \int_{\partial D} \left[\epsilon^+ \frac{\partial G_+}{\partial \n(\x)}
- \epsilon^- \frac{\partial G_-}{\partial \n(\x)}\right]  (\e \times \n)(\y)\; ds(\y) \nonumber  \\
+ \int_{\partial D} \left[\gradx (\epsilon^+ G_+ -\epsilon^- G_-)\right]
 (\e \times \n)(\y)  \cdot [\n(\x) - \n(\y)] \; ds(\y) \nonumber \\
- \i \omega     \int_{\partial D} \left[\epsilon^+ \mu^+ G_+ - \epsilon^- \mu^- G_-\right]
 (\h \times \n)(\y) \times \n(\x) \;ds(\y) \nonumber \\
-  \f{i}{\omega}  \int_{\partial D} 
\left[\gradx (G_+ - G_-) \times \n(\x)\right]   (\divs \h \times \n)(\y)\;    ds(\y) \Big \} \nonumber  \\
=   \frac{2}{\epsilon^++\epsilon^-} \f{i}{\omega}  \int_{\partial D} 
\left[\gradx (G_+ - G_-) \times \n(\x)\right]   (\divs \h \times \n +  i \omega \epsilon^+ \e_p \cdot  \n)(\y)\;    ds(\y), 
%= O(\xi_p^{-1}),%\frac{2 \epsilon^+}{\epsilon^++\epsilon^-}  (\e_{\inc} \times \n)(\x),
\label{neweq_1} 
\end{align}
which is $O(\xi_p^{-1})$ in $H^{1}$ norm thanks to the smoothing property of  $\gradx (G_+ - G_-) $ mentioned in the proof
of Proposition \ref{M is compact}.  
Similar considerations will lead to 
\begin{align}
(\h \times \n)(\x)  
- \frac{2}{\mu^+ + \mu^-} 
\Big \{   - \int_{\partial D} \left[\mu^+ \frac{\partial G_+}{\partial \n(\x)} 
- \mu^- \frac{\partial G_-}{\partial \n(\x)}\right ] (\h \times \n)(\y)\; ds(\y) \nonumber  \\
+ \int_{\partial D} \left[\gradx (\mu^+ G_+ -\mu^- G_-)\right]
 (\h \times \n)(\y)) \cdot [\n(\x) - \n(\y)] \; ds(\y) \nonumber \\
+ \i \omega     \int_{\partial D} \left[\epsilon^+ \mu^+ G_+ - \epsilon^- \mu^- G_-\right]
 (\e \times \n)(\y)\times \n(\x) \; ds(\y) \nonumber \\
+  \f{i}{\omega}  \int_{\partial D} 
\left[\gradx (G_+ - G_-) \times \n(\x)\right]  (\divs \e \times \n)(\y)\;    ds(\y) \Big \} % \nonumber  \\
=   O(\xi_p^{-1}), %\frac{2 \mu^+}{\mu^+ + \mu^-}(\h_{\inc }  \times \n)(\x),  
\label{neweq_3} 
\end{align}
in $H^{1}$ norm.
We now recognize that the left hand side of (\ref{neweq_1}-\ref{neweq_3}) 
is in the standard M\"uller equation form  for  the unknown
$(\e_p \times \n, \h_p \times \n)$. That equation is well posed for all $\omega >0$,  
and all values  of $\epsilon^\pm$ and
$\mu^\pm$ considered in this paper,  as proved in \cite{muller:book}, chapter VI.
We have thus found that $\e_p \times \n$ and $\h_p \times \n$ are $O(\xi_p^{-1})$,
in $H^{1}$ norm.
It follows that $\divs (\h_p \times \n)$ is $O(\xi_p^{-1})$,
in $L^2$ norm, and because of   (\ref{max cond is small}), 
$\e_p \cdot  \n$ is  $O(\xi_p^{-1})$, in $H^{-1}$ norm.\\
Taking the inner product of the terms in equation (\ref{hom eq in e dot n}) with 
$\e_p \cdot \n $ on
$\p D$ yields
\bean
\langle\e_p \cdot \n, (I + \bK_1 + \xi \omega^2 \epsilon^+ \bS) \e_p \cdot \n \rangle = \nonumber \\
   \frac{2 \epsilon^-}{\epsilon^++\epsilon^-}  \langle \e_p \cdot \n,  - \i \omega  \int_{\partial D} \left[\mu^+G_+-\mu^-G_-\right] 
(\h_p \times \n)(\y)  \cdot \n(\x)\; ds(\y) \rangle \nonumber \\
+    \frac{2 \epsilon^-}{\epsilon^++\epsilon^-}
\langle \e_p \cdot \n,  \int_{\partial D}   \left[\gradx (G_+ -G_-) \times \n(\x)\right] \cdot (\e_p \times \n)(\y) \; ds(\y) \rangle \nonumber \\
 - i\xi_p \omega \langle \e_p \cdot \n, \bS \divs (\h_p \times \n) \rangle
\label{for estimate}, % last minus sign is due to anti-linearity in the second argument
\eean
where $\bK_1$ is as in~(\ref{opform e dot n}).
The first two terms of the right hand side of (\ref{for estimate}) are clearly of order
$O(\xi_p^{-1})$ in  $H^1$ norm since $\e_p \times \n$ and $\h_p \times \n$ are $O(\xi_p^{-1})$
in  $H^{1}$ norm. For the third term, using that $\bS$ is self adjoint, 
\bea
\xi_p \omega \langle \e_p \cdot \n, \bS \divs (\h_p \times \n) \rangle =
\xi_p \omega \langle \bS (\e_p \cdot \n),  \divs (\h_p \times \n) \rangle,
\eea
and  since
$\e_p \cdot \n $ is $O(\xi_p^{-1})$ in  $H^{-1}$ norm
we have $\bS ( \e_p \cdot \n)$
is $O(\xi_p^{-1})$ in  $L^2$ norm.
Finally, since $\divs (\h_p \times \n) $ is $O(\xi_p^{-1})$ in  $L^2$ norm
we deduce that
$\xi_p \omega \langle \bS (\e_p \cdot \n),  \divs (\h_p \times \n) \rangle$ is of order
$O(\xi_p^{-1})$. In summary, the right hand side of~(\ref{for estimate}) is of order
$O(\xi_p^{-1})$.
However, Theorem \ref{norm inv lemma} establishes that
the real part of the left hand side
of~(\ref{for estimate}) is bounded below by
$\f12 \| \e_p \cdot \n \|_{L^2(\p D)}^2$
for all $p$ sufficiently large, so that
$\e_p \cdot \n$ is of order $O(\xi_p^{-1/2})$ in the $L^2$ norm.
We have thus proved that $(\e, \h)$ converges strongly to zero in $X$, which
contradicts (\ref{norm1}). \\
To show the uniform bound for the operator $(\Ima + \Mm + \xi \Jm)^{-1} $
for all $\xi$ sufficiently large,
we can repeat the same argument by contradiction
by assuming (\ref{norm1}) and in place of (\ref{p eq}) we assume
\bea
\lim_{p \ri \infty}(\Ima + \Mm + \xi_p\Jm) (\e_p, \h_p) =0.
\eea
Finally, Proposition \ref{prop discrete}
proves that $\Ima + \Mm + \xi \Jm$ 
is invertible 
for all $\xi \in \CC \setminus Y$ where $Y$ is a discrete set.
\end{proof}

We now remark that
in practice the forcing term in equation
(\ref{new eq}) is  smooth since it is derived from an incident field. 
%the incident field is smooth. The forcing term in equation
%(\ref{new eq}) is also smooth,  
 It is  thus useful to study  the  regularity properties of equation (\ref{new eq}).
To that effect we define for $s>0$ the functional space
\bea
 X^s =\{ (\e ,\h): \e\times \n,  \h \times \n \, \in H^{s+1}(\p D)^3,  \e\cdot \n,  
\h \cdot \n \, \in H^s(\p D)
\}
\eea
with the norm $\| \,  \|_{X^s}$ defined by 
\bea
\| (\e,\h) \|_{X^s}^2 = 
  \|\e\times \n\|_{H^{s+1}(\p D)^3}^2 +  \|\h \times \n \|_{H^{s+1}(\p D)^3}^2 +   \|\e\cdot \n
	\|_{H^s(\p D)}^2 + \|\h \cdot \n\|_{H^s(\p D)}^2.
\eea
Identical arguments to those used in the proofs of Propositions~\ref{M is compact} and~\ref{J is compact}
can be repeated to show that
$\Mm, \Jm: X^s \to X^s$  are compact. In fact
 $\Ima + \Mm + \xi \Jm: X^s \to X^s$ is invertible with continuous inverse 
for the same values of $\xi$ as in the statement of Theorem~\ref{maintheorem}.
These results are summarized in the next theorem.

\begin{thm}
Suppose that $(\e^i, \h^i) \in X^s$.
%\bea  (\Ima + \Mm + \xi \Jm) (\e, \h) = (\e^i, \h^i)
%\eea
Then for all $\xi \geq \xi_0$, or
$\xi \in \CC \setminus Y$ for some discrete set $Y$,
equation~(\ref{new eq}) is uniquely solvable. 
The solution $(\e,\h) \in X^s$, and $\| (\e,\h) \|_{X^s}$ 
is linearly bounded by
$\|(\e^i, \h^i)\|_{X^s}$.
The bound is uniform for all $\xi$ such that  $\xi \ge \xi_0$.
\end{thm}

%\appendix
\section{Appendix} \label{append}
\begin{lem} \label{proj}
Let 
$H$ be a separable Hilbert space and
$\{ e_j \, : \, j \in \NN\}$ be a Hilbert basis of $H$.
Let $P_m$
be the orthogonal projection onto $\mathrm{span} \{e_1, ..., e_m\}$,
and $Q_m = I - P_m$. 
Let  $K: H \to H$ be a  linear  compact operator.
Then $Q_m K$ %converges to zero in operator norm, as $m \ri \infty$.
 and $K Q_m$ converge to zero in operator norm, as $m \ri \infty$.
\end{lem}

\begin{proof}
Arguing by contradiction, suppose that $Q_m K$ does not converge to zero in operator norm.
Then there is a positive $\alpha$ and a sequence $(x_m) \in H$ such that $\| x_m \| =1$
and $ \| Q_m K x_m \| \geq \alpha$.
Since $K$ is compact, there
is  a subsequence
  $(K x_{m_p})$ of $(K x_m)$ which converges to $K x$, for some $x$ in $H$.
Clearly, $ Q_{m_p} K x$ converges to zero. 
Now
\bea
 \|Q_{m_p}  K x_{m_p} \| \leq \| Q_{m_p} K x \| + \| Q_{m_p}  (K x_{m_p}  -K x) \|
\leq  \| Q_{m_p} K x \| + \|  K x_{m_p}  -K x \|.
\eea 
Thus $Q_{m_p}  K x_{m_p} $ converges to zero as $m_p \to \infty$, which is
a contradiction.\\
Note that
$Q_m K^* = Q^*_m K^* = (K Q_m)^*$,
using $Q^*_m = Q_m$.
Since in the result above $K$ is arbitrary,
we deduce that $Q_m K^*$ also converges to zero.
\end{proof}

\begin{thm}\label{norm inv lemma}
Let 
$H$ be a separable Hilbert space and
$M,J: H \to H$ be two linear and compact operators.
Assume that 
$J$ is  injective, self adjoint, and positive. 
Then  there is a positive $\xi_0$ such that $I+ M + \xi J$   is invertible for all $\xi>\xi_0$
and $\| (I +M + \xi J)^{-1} \| $ is uniformly bounded. 
In addition $\xi_0$ can be chosen such that for all $\xi>\xi_0$ 
and $x$ in $H$, 
\bea  \mbox{Re } \langle (I +M + \xi J)x, x \rangle \geq \ds \f12 \| x\|^2. \eea
\end{thm}

\begin{proof}
Let $\lambda_1 \geq \lambda_2 \geq ... \lambda_j \geq ... >0$
be the eigenvalues of $J$. Let $e_1, ..., e_j, ...$ be  associated eigenvectors with  norm 1.
They form an orthonormal Hilbert basis of $H$, since $J $ is injective. \\
Let $P_m$ be the orthonormal projection on  span $\{ e_1, ..., e_m \}$ and $Q_m = I - P_m$.
%Fix $\alpha >0$. 
By Lemma \ref{proj}, we can pick $m$ large enough such that $ \| Q_m M \| < \f14$. % (possible due to lemma). \\ %  and $\|  M  Q_m\| < \alpha$.\\
For $\xi >0$ we write
\bea
\mbox{Re} \langle  (I+ M +  \xi J )  x,  x \rangle = t_1 + t_2+ t_3 + t_4 ,
\eea 
with
\bea
t_1 &= & \| x\|^2, \\ %<P_m U x, P_m Ux > +  <Q_m U x, Q_m Ux > ,\\
        t_2 &=  &\xi \mbox{Re }\langle P_m J  x , P_m x \rangle +\mbox{Re } \langle P_m M x,  P_m x \rangle ,\\
				t_3 &=&  \xi \mbox{Re} \langle Q_m J  x , Q_m x \rangle, \\
				t_4 &=& \mbox{Re} \langle Q_m M  x,  Q_m  x \rangle. %<Q_m M U x,  P_m U x> + <P_m M U x,  Q_m U x> + <Q_m M U x,  Q_m U x>.
\eea
As $J$ commutes with $P_m$ and $Q_m$
\bea
 t_2 &=& \xi  \langle J  P_m x , P_m x \rangle + \mbox{Re } \langle P_m M  x,  P_m  x \rangle ,\\
				t_3 &=& \xi  \langle JQ_m x , Q_m x \rangle.
\eea
As $J$ is injective, self-adjoint, coercive, and the range of $P_m $ is finite dimensional, there is a positive $C$ such that 
$\langle J  P_m x , P_m x \rangle \geq C \| P_m x\|^2$.
It follows that
\bea
t_2 &\geq &\xi  C \| P_m x\|^2 - \| M \| \| x \| \| P_m x\| \\
&\geq & \xi  C \| P_m  x\|^2 - \f12 \f{1}{\beta}\| M \|^2 \| P_m x\|^2  - \f12 \beta \| x \| ^2,
\eea
for any $\beta >0$.
We note that $t_3 \geq 0$ and $t_4 \geq - \f14\|  x\|^2$.
We  set $\beta = \f12 $, $\xi_0 = \f{\| M \|^2}{C}$ to find, for any $\xi  \geq \xi_0$  
\bea
\mbox{Re } \langle (I+ M + \xi  J ) x,  x \rangle  \geq \f12 \| x\|^2.
\eea
%It follows that $I + M + s J^*J $ is injective for this choice of $s$ and as $M$ and $J$ are compact, $I + M + s_0 J^*J $ is invertible
%with continuous inverse.
By the Lax Milgram theorem, this estimate shows that $(I+ M + \xi  J )$ is invertible for all
 $\xi  \geq \xi_0$.
As $ \mbox{Re} \langle (I+ M + \xi  J )(I+ M + \xi  J)^{-1} x,  (I+ M + \xi J)^{-1} x \rangle  $
is greater or equal than $\ds \f12  \| (I+ M + \xi  J)^{-1} x \|^2$  and less or equal than 
$\| x \| \, \| (I+ M + \xi  J)^{-1} x\|$
for all $x$, it follows that the norm of $(I+ M + \xi  J)^{-1}$ is less or equal than 2.
\end{proof}

\begin{prop}\label{isometry lemma}
Let 
$H$ be a separable Hilbert space 
and $M,J: H \to H$ be two linear and compact operators.
%and $M$ and $J$ two linear and compact operators from $X$ to $X$.
Assume that 
$J$ is  injective, self-adjoint, and positive.  
Then  $I+ M + \xi J$   is invertible, with continuous inverse, except possibly for $\xi$ in a discrete set of $\CC$.
\end{prop}

\begin{proof}
The result follows from Proposition \ref{prop discrete} and Theorem \ref{norm inv lemma}. 
%The result follows from Proposition \ref{prop discrete} and Theorem \ref{norm inv lemma}. 
\end{proof}

\begin{thm} \label{all symmetric} 
Let $H$ be a separable Hilbert space, 
and $M,J : H \to H$ be two  linear  compact operators.
Assume $J$ is self-adjoint and non-negative,
 $N(J)$ is invariant under $M$, and the intersection of $N(J)$ and $N(I+M)$ is trivial.
Then $I + M  + \xi J $ is invertible for all real $\xi > \xi_0$, for some constant $\xi_0$ and for 
all $\xi \in \CC \setminus Y $ for some  discrete subset $Y \subseteq \CC$.
%all $\xi$ outside a discrete set  of $\CC$.
\end{thm}

\begin{proof}
Let  $P$ denote the orthogonal projection on $N(J)$.
The operator   $I + M  + \xi J $ can be re-written in blocks
%\begin{align}
%I + M  + \xi J = &(I -P)(I + M  + \xi J) (I -P)  &\\
 %                         &   PM(I-P)   & + P(I + M  )P 
%\end{align}
\bea
I + M  + \xi J = (I -P)(I + M  + \xi J) (I -P) & +
                             PM(I-P)   & + P(I + M  )P .
\eea
We note that due to the assumptions on $M$ and $P$, $ P(I +M)P$ is an invertible
operator from $N(J)$ to $N(J)$, with continuous inverse. % $Q$.
Since we can apply Theorem \ref{norm inv lemma} to the block
$(I -P)(I + M  + \xi J) (I -P)$, the result is proved.
\end{proof}

\begin{prop} \label{injective counter example}
Let $H$ be a separable Hilbert space. There exist two compact linear operators $M,J: H \to H$
such that $J$ is injective and $I + M + \xi J$ is singular for all $\xi \in \CC$.
%Let $H$ be a separable Hilbert space. There exist two compact linear operators $M$ and $J$ from $H$ to $H$ 
%such that $J$ is injective and $I + M + \xi J$ is singular for all $\xi$ in $\CC$.
\end{prop}

\begin{proof}
Let $H$ be a separable Hilbert space over $\CC$ of infinite dimension with Hilbert basis
 $\{e_1, e_2, ... \}$. 
%Define $M$ as follows:
%\begin{align*}
%M(e_1) := -e_1 \\
%M(e_k) := 0 \; \; \forall \; \; k \ne 1\\
%\end{align*}
We can define a compact linear operator $M$ by setting $M e_1 = - e_1$, and $M e_ k =0$, for all $k \geq 2$.
%Since $M$ is a finite rank operator, it is compact.
 Now we define a compact linear operator $J$ by setting $J e_k = \ds \frac{e_{k+1}}{k+1} $ for all $k$. 
%as follows:
%\begin{align*}
%$J(e_k)& := \frac{e_{k+1}}{k+1} \; \; \forall k \in \N \\
%J_p(e_k) &:= J(e_k) \; \; \forall k \le p,  \; \;  J_p(e_k) := 0  \; \; \forall k > p\\
%\end{align*}
%We show that $J_p \ri J$ (in operator norm) as $p \ri \infty$ so that $J$ is the limit of finite rank operators and thus is compact. Indeed if $||x|| \le 1$, we write $x = \sum_{k=1}^{\infty} \langle x, e_k \rangle e_k$ and hence
%\begin{align*}
%(J-J_p)x &= \sum_{k=p+1}^{\infty} \frac{\langle x , e_k \rangle e_{k+1}} {k+1}\\
%||(J-J_p)x||^2 & = \sum_{k=p+1}^{\infty} \frac{| \langle x , e_k \rangle |^2 }{(k+1)^2}\\
 %                     & \le \frac{1}{(p+1)^2} \sum_{k=1}^{\infty} | \langle x , e_k \rangle |^2 \\
   %                   & = \frac{1}{(p+1)^2} ||x|| \\
     %                 & \le \frac{1}{(p+1)^2} \\
%\end{align*}
%And thus $||J-J_p|| \le \frac{1}{(1+p)}$ so $J_p \ri J$. 
Since for all $x$ in $H$, %$\ds x = \sum_{k=1}^{\infty} \langle x , e_k \rangle e_k$,
%
%\begin{align*}
%Jx & = \sum_{k=1}^{\infty} \frac{\langle x , e_k \rangle e_{k+1}} {k+1}\\
$ \| Jx  \|^2  = \ds  \sum_{k=1}^{\infty} \frac{|\langle x , e_k \rangle|^2 } {(k+1)^2} $, %\\
%\end{align*}
%And thus $Jx = 0 $ and $ \langle x, e_k \rangle = 0  \; \; \forall \; \; k \in \N$. Since $x = \lim_{N \ri \infty} \sum_{k=1}^{N} \langle x , e_k \rangle e_k$ we have $Jx = 0 $ and $ \langle x, x \rangle = 0$ so that $J$ is indeed injective. 
 $J$ is injective. \\
%Now fix $s \in \CC \setminus \{ 0 \}$ and define
Now let $\xi \in \CC$ and set
\begin{equation*}
u = \sum_{k=1}^{\infty} \frac{(-\xi)^{k-1} e_k}{k!}.
\end{equation*}
Note that  $\| u\|^2 \geq 1$.
%\begin{align*}
%||u||^2 = \sum_{k=1}^{\infty} \frac{|s|^{2k}}{(k!)^2} \ge| s|^2 > 0
%\end{align*}so that $u \ne 0$. 
Moreover,
\bea
(I + M) u  = \sum_{k=2}^{\infty} \frac{(-\xi)^{k-1} e_k}{k!} 
               = \sum_{k=1}^{\infty} (-1)^{k}\frac{\xi^{k}  e_{k+1}}{(k+1)!}.
\eea
and 
\bea
\xi Ju  = \xi \sum_{k=1}^{\infty} \frac{(-\xi)^{k-1} J  e_{k}}{k!} 
= - \sum_{k=1}^{\infty} (-1)^{k}\frac{\xi^{k}  e_{k+1}}{(k+1)!},
         %& = \sum_{k=1}^{\infty} \frac{s^{k+1} (-1)^k e_{k+1}}{(k+1)!}\\
         %&= -(I+M)u\\
\eea
Thus $(I +M + \xi J)u = 0$, for all $\xi \in \CC$.
%. This shows that for each $s \in \C \setminus \{0\}$ the operator $I + M + sJ$ is \textit{not} invertible.
\end{proof}

\end{document}